\documentclass[12pt]{amsart}

\usepackage{amsmath,amsthm,amsfonts,amssymb,bm,graphicx}
  \usepackage[parfill]{parskip}
\usepackage{hyperref, mathrsfs}
\hypersetup{colorlinks=true,citecolor=blue,linkcolor=blue,urlcolor=blue,
pdfstartview=FitH, pdfauthor=Gonek--Ivi\'c ,pdftitle=On the distribution of positive and
 negative values of Hardy's $Z$-function}
 \numberwithin{equation}{section}
   \theoremstyle{plain}
\newtheorem{thm}{Theorem}

\newtheorem*{con* A}{Conjecture A}
\newtheorem*{con* B}{Conjecture B}

\newtheorem{lem}{Lemma}
\newtheorem{thm*}{Theorem}
\newtheorem*{con*}{Conjecture}
\newtheorem*{lem*}{Lemma}

\theoremstyle{definition}

\theoremstyle{definition}

\renewcommand{\geq}{\geqslant}
\renewcommand{\leq}{\leqslant}

\newcommand{\be}{\begin{equation}}
\newcommand{\ee}{\end{equation}}

\def\hf{{\textstyle{\frac12}}}
\def\a{\alpha}\def\b{\beta}
\def\d{{\,\rm d}}
\def\e{\varepsilon}

\def\G{\Gamma} \def\g{\gamma}

\def\z{\zeta}

\def\={\;=\;}
\def\le{\leqslant}

\def\zt{\zeta(\hf+it)}

\def\leq{\leqslant}
\def\geq{\geqslant}

\makeatletter
\DeclareRobustCommand\widecheck[1]{{\mathpalette\@widecheck{#1}}}
\def\@widecheck#1#2{%
    \setbox\z@\hbox{\m@th$#1#2$}%
    \setbox\tw@\hbox{\m@th$#1%
       \widehat{%
          \vrule\@width\z@\@height\ht\z@
          \vrule\@height\z@\@width\wd\z@}$}%
    \dp\tw@-\ht\z@
    \@tempdima\ht\z@ \advance\@tempdima2\ht\tw@ \divide\@tempdima\thr@@
    \setbox\tw@\hbox{%
       \raise\@tempdima\hbox{\scalebox{1}[-1]{\lower\@tempdima\box
\tw@}}}%
    {\ooalign{\box\tw@ \cr \box\z@}}}
\makeatother

\begin{document}

 \title{On the distribution of  positive and negative values of Hardy's $Z$-function}
 \author{Steven M. Gonek}
 \author{Aleksandar Ivi\'c}

\email{gonek@math.rochester.edu}
\address{Department of Mathematics, University of Rochester, Rochester, NY 14627}

\email{aleksandar.ivic@rgf.bg.ac.rs, aivic\_2000@yahoo.com}
\address{Serbian Academy of Sciences and Arts, Knez Mihailova 35,
11000 Beograd, Serbia}

\thanks{Work of the first  author was partially supported by NSF grant DMS-1200582.}

\keywords{Riemann zeta-function, Hardy's function,  distribution of positive and negative values}

\begin{abstract}
We investigate the distribution of  positive and negative values of Hardy's function
$$
Z(t) := \zt{\chi(\hf+it)}^{-1/2}, \quad \z(s) = \chi(s)\z(1-s).
$$
In particular we prove that
$$
\mu\bigl(I_{+}(T,T)\bigr)  \;\gg T\; \qquad \hbox{and}\qquad  \mu\bigl(I_{-}(T, T)\bigr) \; \gg \; T,
$$
where $\mu(\cdot)$ denotes the Lebesgue measure and
\begin{align*}
 {  I}_+(T,H) &\;=\; \bigl\{T< t\le T+H\,:\, Z(t)>0\bigr\},\\
 {  I}_-(T,H) &\;=\; \bigl\{T< t\le T+H\,:\, Z(t)<0\bigr\}.
\end{align*}

\end{abstract}

\subjclass[2010]{Primary 11M06}

\setcounter{tocdepth}{2}  \maketitle 
  \maketitle

\section{Introduction and statement of results}

 Hardy's function $Z(t)$ is defined as
\be\label{defn: Z}
Z(t) := \zt{\chi(\hf+it)}^{-1/2},
\ee
where $\chi(s)$ ($s\in \mathbb C$) is the factor from the functional equation for $\z(s)$,
namely, $\z(s) = \chi(s)\z(1-s)$. Thus
$$
\chi(s) = 2^s\pi^{s-1}\sin(\hf \pi s)\G(1-s),\quad
\chi(s)\chi(1-s)=1.
$$
(See  the second author' s  monograph \cite{Iv4} for an extensive account of the $Z$-function.)
It follows  that
$$
\overline{\chi(\hf + it)} = \chi(\hf-it)= \chi^{-1}(\hf+it),
$$
so that $Z(t)\in\mathbb R$ when $t\in\mathbb R$, $Z(t) = Z(-t)$,
and $|Z(t)| =|\zt|$. Thus the zeros of $\z(s)$ on the ``critical line'' $\Re s =1/2$
correspond to the real zeros of $Z(t)$, which makes $Z(t)$ an invaluable tool
in the study of the zeros of the zeta-function on the critical line.

Our main interest here  is in the distribution of positive and negative values of $Z(t)$, a topic previously discussed in \cite{Iv3},  Chapter 11 of \cite{Iv4}, and in \cite{Iv5}.
If one looks at  graphs of $Z(t)$ in various $t$ ranges,  it is difficult to detect a  bias toward
positive or negative values.
Let $2\leq H \leq T $ and set
\begin{align*}
 {  I}_+(T,H) \;=\; \bigl\{T< t\le T+H\,:\, Z(t)>0\bigr\}
\end{align*}
and
\begin{align*}
{I}_-(T,H) \;=\;  \bigl\{T< t\le T+H\,:\, Z(t)<0\bigr\} .
\end{align*}
Also let $\mu(\cdot)$ denote Lebesgue  measure.
 Mathematica   calculations of  $I_+(T,H)$ and $I_{-}(T,H)$ for divers values
of $H$ and $T$ suggest the conjecture that the measure of   these  sets is approximately $ H/2$, even when $H$ is  quite small relative to $T$ (see  Tables 1 and 2 below).  The purpose of this paper is to lend theoretical support  to this conjecture by showing that $Z(t)$ takes
positive and negative values a positive proportion of the time on intervals that are not too short.

\begin{thm}\label{thm: uncond}
We have
\be\label{uncond lwr bds}
\mu\bigl(I_{+}(T,T)\bigr)  \;\gg T\; \qquad \hbox{and}\qquad  \mu\bigl(I_{-}(T, T)\bigr) \; \gg \; T .
\ee
\end{thm}

Our method would also allow us to prove that  $\mu(I_{\pm}(T, H))\gg H$ for $H$ somewhat smaller than $T$.
Moreover, with more effort we could replace the $\gg$  symbols by  explicit inequalities.
However, a heuristic argument suggests that the values  we would obtain for the  constants,
 even using the best currently available mean value estimates, would  be rather small,
 so we have  not bothered to calculate them. Note also that
 it follows from \eqref{uncond lwr bds} that $\mu (I_{\pm}(0,T) ) \gg T$.

By a  different argument we can prove  a conditional result
with reasonably good   constants.

\begin{thm}\label{thm: cond 1}
Assume the Riemann hypothesis and Montgomery's pair correlation conjecture are true.
Then for all $T$ sufficiently large we have
\be\label{cond lwr bds 1}
\mu\bigr(I_{+}(0,T)\bigl)  \geq .32909\,T \qquad \hbox{and}\qquad  \mu\bigl(I_{-}(0, T)\bigr)   \geq .32909\,T.
\ee
\end{thm}

The well-known Riemann hypothesis is the statement that all complex zeros of $\z(s)$ have real parts
equal to 1/2, and for a formulation of Montgomery's pair correlation conjecture, see \cite{Mon} and
\eqref{PC}.

It is worth noting that the answer to the corresponding  question for
$\log|Z(t)| =\log|\zeta(\frac12+it)|$, that is, how often $\log|Z(t)|$ is positive and how often it is negative,  is known.  For
Selberg~\cite{Sel 1} (also see   Tsang~\cite{Ts}) has shown that \
 $ \log|\zeta(\frac12+it)|/(\frac12 \log\log t)^{1/2}$ \
is normally distributed with mean zero.
Thus, the measure of the  set of $t\in[T,2T]$ for which $\log|Z(t)|$ is either positive or
negative is $\sim T/2$ as $T\to\infty.
$\footnote{From Tsang's version of the result one can deduce that
 the measure of the set of $t\in[T,T+H]$ for which $\log|Z(t)|$
 is either positive or negative is $\sim H/2$, where $T^{1/2+\e}\leq H\leq T$ and $0<\e\leq 1/2$.}

\begin{table}[h]\label{table: dyadic}
\centering 
\begin{tabular}{c c} 
\hline\hline 
$T$ & $\mu(I_+(T, T))/(\frac12 T)$ \\ [0.5ex] 
\hline\hline
100 & 0.943850
\\ 
200 & 0.987534 \\
500 &  0.963277 \\
1000& 0.981253 \\
5,000 & 0.986981  \\
10,000 & 0.990367 \\
50,000 &  0.968667     \\[1ex] 
\hline 
\end{tabular}
 \caption{Ratios of the measures of sets in dyadic intervals where $Z(t) > 0$ to the conjectured values.}
\end{table}
\vskip.15in
 \begin{table}[h]\label{table: length$100$}
\centering 
\begin{tabular}{c c} 
\hline\hline 
$T$ & $\mu(I_+(T,  100))/50$ \\ [0.5ex] 
\hline\hline 
100 & 0.943850
\\ 
200 & 0.989211 \\
500 &  0.967649 \\
1000& 0.876483 \\
5,000 & 1.04117  \\
10,000 & 0.967802 \\
100,000 & 1.05694\\
1,000,000 & 0.959324\\
10,000,000 & 1.00084 \\
100,000,000 & 1.00168 \\
[1ex] 
\hline 
\end{tabular}
 \caption{Ratios of the measures of sets in intervals of length $100$ where $Z(t)>0$ to the conjectured value.}
\end{table}
\vskip.15in



 \section{Lemmas  for the Proof of Theorem~\ref{thm: uncond}}

In this section we set down the lemmas necessary  for the proof of
Theorem~\ref{thm: uncond}.

Define the arithmetic function $\a_\nu$ by
  $$
 \frac{1}{\sqrt{\zeta(s)}}=\sum_{\nu=1}^\infty \a_\nu \nu^{-s} \qquad (\Re s >1).
 $$
For $1\leq \nu\leq X$ let
$$
\b_\nu=\a_\nu\Big(1-\frac{\log \nu}{\log  X}\Big)
$$
and set
 $$
B_X(s)=\sum_{\nu\leq X} \b_\nu \nu^{-s}.
 $$
In his famous proof that a positive proportion of the zeros of the zeta-function are on the critical line,
Selberg~\cite{Sel} used  $|B_X(\hf+it)|^2$ to mollify (smooth) $Z(t)$.
This function  serves the same purpose  for us here.

\begin{lem}\label{lem: little o mean}
Let   $X=T^\theta$ with $0<\theta<1/4$.
Then
\be\label{o mean}
\int_T^{2T}  Z(t) |B_X(\hf+it)|^2 \d t\; =\; o(T) \qquad(T\to\infty) .
\ee
\end{lem}

\begin{proof}
By \eqref{defn: Z} we can write the integral in question as
\be\label{chi zeta BB int}
\frac1i \int_{\frac12+iT}^{\frac12+2iT}  \chi(s)^{-1/2} \zeta(s) B_X(s) B_X(1-s)\d s.
\ee
By Cauchy's theorem we may replace  the segment of integration
$[\frac12+iT, \,\frac12+2iT]$
  by the other three sides of the rectangle with vertices
  $\frac12+iT, c+iT,\,c+2iT$, and $\frac12+2iT$, where
  $c=1+1/\log T$,   traversed in that  order.
It is not difficult to see that the coefficients of $B_X(s)$ satisfy
 $ |\beta_\nu|\leq 1$ (since $\alpha_\nu$ is multiplicative and $0 \le 1 - \log\nu/\log X \le1$).
Thus for $\sigma\geq -1$  we have
\be\label{B bd}
B_X(s) \;\ll\; \max (X^{1-\sigma}, \log X).
\ee
Moreover, for $\sigma  \geq \hf, \,  t\geq 2$,
\be\label{zeta bd}
\zeta(\sigma+it) \;\ll\;  (t^{(1-\sigma)/3}+1)\log t,
\ee
which is the standard convexity bound for $\z(s)$ and follows from $\zt \ll t^{1/6}\log t$
and $\z(1+it) \ll \log t$.
Also for $-1\leq \sigma\leq 2, \ t\geq 2$, by Stirling's formula for the gamma-function, we have
\be\label{chi asymp}
\chi(s) = \Big(\frac{2\pi}{t}\Big)^{\sigma+it-1/2}{e}^{i(t+\pi/4)}
 \Big(1+O\Bigl(\frac{1}{t}\Bigr)\Big).
\ee
The contribution of the horizontal sides  of the rectangle, $[\frac12+iT, c+iT]$ and
$[\frac12+2iT, c+2iT]$, is therefore
\begin{align*}
  \ll \int_{1/2}^{c} \max(X^{1-\sigma}, \log X) X^\sigma T^{(\sigma-1/2)/2}
 (T^{(1-\sigma)/3}  + 1 ) \log T   \d\sigma
\ll XT^{1/4}\log T.
\end{align*}
On the right-hand  side of the rectangle the series for $\z(s)$ is absolutely convergent.
Therefore, by using \eqref{chi asymp}, we see that the integral over this side  equals
\be\notag
\int_T^{2T} \bigg(\sum_{n=1}^\infty n^{-c-it}\bigg)
\bigg(\sum_{\nu\le X}\b_\nu \nu^{-c-it} \bigg)
\bigg(\sum_{\mu\le X}\b_\mu \mu^{c-1+it} \bigg)
\left(\frac{t}{2\pi}\right)^{ (c- {1}/{2} +it)/2}e^{-i(t+\pi/4)/2}
\Big(1+O\Bigl(\frac{1}{t}\Bigr)\Big)\d t.
\ee
Using
 $$
\Bigl|\sum_{n=1}^\infty  n^{-c-it }\Bigr| \le
\sum_{n=1}^\infty  n^{-c } = \z\left(1 + \frac{1}{\log T}\right) \ll \log T
$$
together with \eqref{B bd} and \eqref{zeta bd}, it is seen that
  the $O$-term contributes $O( T^{1/4} X \log^2 T).$
The remaining expression is
\begin{align*}
&
e^{-\pi i/8}\sum_{n=1}^\infty \sum_{\nu  \le X} \sum_{\mu\le X}
 \frac{ \b_\nu \b_\mu  \mu^{c-1}}{n^{c } \nu^{c }}
\int_T^{2T}
 \left(\frac{t}{2\pi}\right)^{ (c-1/2)/2}e^{i(t/2)\log (t \mu^2/2\pi en^2\nu^2)}
 \d t.
\end{align*}
 By the second derivative bound for exponential integrals
 (see  Lemma 2.2 of \cite{Iv1} or Lemma 4.5 of \cite{Tit})  the integral is
  $ \ll  T^{3/4}.$  Therefore the entire expression  is
$$
\ll   T^{3/4 } \sum_{n=1}^\infty  n^{-c } \sum_{\nu \le X} \nu^{-c } \sum_{\mu \le X}
   \mu^{c-1}  \ll T^{3/4 }X \log^2T.
$$
Combining our estimates, we find that the integral in \eqref{chi zeta BB int}
is
$$
\ll XT^{1/4}\log T + T^{1/4} X \log^2 T+T^{3/4 }X \log^2T \ll T^{3/4 }X \log^2T .
$$
Thus, if we take $X=T^\theta$ with $\theta < 1/4$, \eqref{o mean} follows.

\end{proof}

\begin{lem}\label{lem: moll lower bd}
Let   $X=T^\theta$ with $0<\theta<1/2$.
Then
\be\label{asymp}
\int_T^{2T}  |Z(t)| |B_X(\hf+it)|^2\d t \;\geq\;  T + o(T).
\ee
\end{lem}

\begin{proof}
We begin by noting that
\be\notag
\int_T^{2T}  |Z(t)| |B_X(\hf+it)|^2\d t
=\int_T^{2T}  |\z(\hf+it)| |B_X(\hf+it)|^2 \d t
\geq \bigg|  \int_T^{2T}  \z(\hf+it) B_X(\hf+it)^2 \d t \bigg|.
\ee
By the  well known approximate formula (see Chapter 1 of \cite{Iv1})
\begin{align*}
\zt = &\sum_{n\le T}n^{-1/2-it} - \frac{T^{1/2-it}}{ 1/2-it} +O(T^{-1/2})
 \qquad \qquad (|t|\leq 2T) ,
\end{align*}
and for  the range $T\leq t\leq 2T$ this becomes
\begin{align*}
\zt = &\sum_{n\le T}n^{-1/2-it}  +O(T^{-1/2}) .
\end{align*}
Hence,
$$
\int_T^{2T}  |Z(t)| |B_X(\hf+it)|^2\d t
\geq
\bigg|  \int_T^{2T}  \bigg(\sum_{n\le T}n^{-1/2-it}  +O(T^{-1/2}) \bigg)
B_X(\hf+it)^2 \d t \bigg|.     $$
By the mean value theorem for Dirichlet polynomials and since
$|\b_\nu|\leq 1$, the $O$-term contributes
$$
\ll  T^{-1/2} \sum_{\nu\leq X}\frac{\b_\nu^2}{\nu}(T+ \nu)
\ll  T^{1/2}\log X +T^{-1/2}X \ll T^{1/2}\log T
$$
for $X=T^\theta$ with $\theta<1/2$.
To treat the other term let
$$
B_X(s)^2 = \sum_{m\leq X^2} b(m) m^{-s},
$$
where $b(m)=\sum_{d|m}\b_d \b_{m/d}$. Note that $b(1) =1$ and
$|b(m)|\leq d(m)$, the divisor function of $m$. Thus, we find that
\be\notag
\begin{split}
\int_T^{2T}  \bigg(\sum_{n\le T}n^{-1/2-it}   \bigg) B_X(\hf+it)^2 \d t
=\,&\int_T^{2T}  \bigg(\sum_{n\le T}n^{-1/2-it}   \bigg)
 \bigg(\sum_{m\le X^2} b(m)m^{-1/2-it}   \bigg) \d t \\
 =\,&T+      \sum_{\substack{n\leq T\\  m\le X^2\\ mn>1     }}   \frac{b(m)}{(mn)^{1/2}}
 \int_T^{2T}    (mn)^{-it}  \d t  \\
  =\,&T+    O\Bigg(  \sum_{\substack{n\leq T\\  m\le X^2\\ mn>1 }}
  \frac{|b(m)|}{(mn)^{1/2}\log mn}\Bigg).
\end{split}
\ee
The $O$-term is
$$
\ll\; \sum_{n\leq T} \frac{1}{n^{1/2}}\sum_{m\leq X^2}\frac{d(m)}{m^{1/2}}
\;\ll\; T^{1/2} X\log X.
$$
Combining our estimates, we find that
$$
\int_T^{2T}  |Z(t)| |B_X(\hf+it)|^2\d t
\geq T+O(T^{1/2} X\log X).
$$
The result now follows provided that $0<\theta<1/2$.
\end{proof}

\begin{lem}\label{lem: Selberg}
Let   $X=T^\theta$ with $0<\theta<1/100$.
Then
\be\label{est: Sel bd}
\int_T^{2T } Z(t)^2   |B_X(\tfrac12+it)|^4 \d t \;\ll\; T.
\ee
\end{lem}
\begin{proof}
This estimate is implicit in the   proof of Lemma 15 of Selberg~\cite{Sel}. His notation   differs from ours, so we shall briefly indicate the differences and describe how to obtain \eqref{est: Sel bd} from  his argument.

Selberg writes
$\eta(t)$ for our $B_X(\frac12+it)$ and $\eta_h(t)$ for   $B_X(\frac12+i(t+h))$.
Moreover, he takes   $\xi =T^{(2a-1)/20}$ with  $1/2<a<3/5$ for the length of $\eta(t)$, whereas we write $X=T^\theta$ for the length of $B_X(s)$.
Instead of $Z(t)$, Selberg works with
$$
 X(t) = - \Big(\frac\pi2\Big)^{1/4} Z(t) \Big(1+O\Big(\frac1t\Big)\Big)
 $$
 for $t$ positive  (see equations (2.1)--(2.3) on  p. 92 of~\cite{Sel}).
 In the course of the proof of   Lemma 15, Selberg estimates   the integral
 $$
  \int_T^{T+U} X(t+h) X(t+k) |\eta_h(t) \eta_k(t)|^2\d t
 $$
 for $0\leq h, k \leq H$, where $H\leq 1/\sqrt{\log \xi}$
 and $T^a\leq U  \leq T^{3/5}$
  (see ~\cite{Sel}, p. 100,
 just below equation (4.3)).
If we take this with $h=k=0$,   we see that
\be\label{Sel mean-our mean}
  \int_T^{T+U} X(t)^2   |\eta (t)|^4\d t
  =\Big(\frac\pi2\Big)^{1/2} \Big(1+O\Big(\frac1T\Big)\Big)  \int_T^{T+U} Z(t)^2
   |B_X  (\tfrac12 +i t)|^4\d t.
 \ee
Now, Selberg~\cite{Sel} (see the  bottom of p. 108) shows that
\be\label{Sel formula}
  \int_T^{T+U} X(t+h) X(t+k) |\eta_h(t) \eta_k(t)|^2\d t
  =\sqrt{2\pi} U K(h-k) +O(T^{1/2} \xi^7),
\ee
 where
 $$
 K(u) =\Re \bigg( \tau^{iu} \sum_{\nu_1, \nu_2, \nu_3, \nu_4 <\xi }
 \frac{\b_{\nu_1} \b_{\nu_2} \b_{\nu_3} \b_{\nu_4}}{\nu_1\nu_2\nu_3 \nu_4}
 \frac{\kappa^{1+iu}}{(\nu_2\nu_3)^{iu}} \sum_{n<\tau \kappa/\nu_2\nu_4} n^{-1-iu}\bigg)
 $$
with $\tau=\sqrt{ T/2\pi}$ and $\kappa=(\nu_1\nu_3, \nu_2 \nu_4)$.
Over the course of the next five pages Selberg proves  that
$K(u)=O(1)$ for $0<u\leq 1/\log \xi$  (see near the bottom of p. 113).
We need this with $u=0$, but that also follows because, as is apparent from its definition,  $K(u)$ is continuous at  $u=0$.
(Selberg excludes $u=0$ because there are poles in an  expression
he  uses to approximate a truncation of the zeta function; see the displayed equation
just after (4.23).)

Taking $1/2<a<3/5$ corresponds to taking $X=\xi =T^\theta$ with $0<\theta=(2a-1)/20< 1/100$.
Then, if $U\gg T^{1/2+7\theta}$, we find that  $U$ satisfies $T^a\leq U\leq T^{3/5}$, as required, and
from \eqref{Sel mean-our mean} and \eqref{Sel formula} we have
\begin{align*}
\int_T^{T+U} Z(t)^2   |B_X(\tfrac12+it)|^4 \d t \;\ll\; U.
\end{align*}
 Splitting the interval $[T, 2T]$  into subintervals of length $U$
 and adding the results, we finally obtain \eqref{est: Sel bd}.
\end{proof}

\section{Proof of Theorem~\ref{thm: uncond}}

We  prove  only  the first estimate in~\eqref{uncond lwr bds} as the proof of the
second is similar.

Clearly we have, setting $I_\pm(T) = I_\pm(T,T)$ for shortness,
\be\label{Z |B^2|}
\int_T^{2T}  Z(t) |B_X(\hf+it)|^2\d t
=\int_{I_+(T)}  Z(t) |B_X(\hf+it)|^2\d t+\int_{I_-(T)}   Z(t) |B_X(\hf+it)|^2\d t,
\ee
and
\be\label{|Z| |B^2|}
\int_T^{2T}  |Z(t)| |B_X(\hf+it)|^2\d t
=\int_{I_+(T)}  Z(t) |B_X(\hf+it)|^2\d t - \int_{I_-(T)}   Z(t) |B_X(\hf+it)|^2\d t.
\ee
Adding  \eqref{Z |B^2|} and \eqref{|Z| |B^2|}, we deduce that
\be\label{add}
\int_{I_+(T)}  Z(t) |B_X(\hf+it)|^2\d t = \frac{1}{2}\left(\int_T^{2T}  Z(t) |B_X(\hf+it)|^2\d t
+ \int_T^{2T}  |Z(t)| |B_X(\hf+it)|^2\d t\right).
\ee
By Lemma~\ref{lem: little o mean}
\be\label{asymp}
\int_T^{2T}  Z(t) |B_X(\hf+it)|^2\d t = o(T)\qquad(T\to\infty),
\ee
and by Lemma~\ref{lem: moll lower bd}
\be\label{small mean}
\int_T^{2T}  |Z(t)| |B_X(\hf+it)|^2\d t   \geq  T + o(T)  \qquad(T\to\infty).
\ee
Thus, by  \eqref{Z |B^2|}--\eqref{small mean} we obtain
\be\label{I_+ lwr bd}
\int_{I_+(T)}  Z(t) |B_X(\hf+it)|^2\d t\geq \hf T + o(T) \qquad(T\to\infty).
\ee
By  the Cauchy-Schwarz  inequality we then deduce that
\be\label{}
\hf T + o(T)  \le \mu\bigl(I_+(T,T)\bigr)^{1/2}{\left(\int_T^{2T}|Z(t)|^2 |B_X(\hf+it)|^4\d t\right)}^{1/2}.
\ee
The first bound in \eqref{uncond lwr bds}   now follows   from (3.7) and  the estimate
$$
\int_T^{2T}|Z(t)|^2|B_X(\hf+it)|^4\d t\;\ll\; T
$$
in Lemma~\ref{lem: Selberg}.

\section{Proof of Theorem~\ref{thm: cond 1}}

In this section we assume both the Riemann hypothesis and Montgomery's pair correlation conjecture.
To state the latter, let $\g, \g'$ denote arbitrary ordinates of zeros of the zeta-function and let
$$
N(T)  = \sum_{0<\g\leq T} 1.
 $$
As is well known (see e.g., \cite{Iv1} or \cite{Tit}),
\be\label{N(T)}
N(T)\sim \frac{T}{2\pi} \log \frac{T}{2\pi} \qquad\quad (T\to\infty).
\ee
Montgomery's pair correlation conjecture~\cite{Mon} asserts
that, if $\a, \b$ are fixed real numbers with $\a<\b$, then
\be\label{PC}
  \sum_{\substack{0<\g,\g' \leq T\\  {2\pi \alpha}/{\log T} \leq \gamma' -\gamma \leq  {2\pi \beta}/{\log T} }}  1\sim
\bigg(\int_\alpha^\beta \biggl[1-\bigg(\frac{\sin\pi u}{\pi u}\bigg)^2\biggr] \d  u +\delta(\a,\b) \bigg) N(T)
\ee
as $T\to \infty$. Here   $\delta(\a, \b) =1$ if $0\in[\a,\b]$ and
$\delta(\a, \b) =0$ otherwise.
We  define
$$\mathcal S_+(T)=  \Bigl\{0<\gamma\leq T :  Z'(\g)>0  \Bigr\}, $$
$$\mathcal S_{-}(T)= \Bigl\{0<\gamma\leq T :  Z'(\g)<0  \Bigr\} ,$$
and  define, with $|\mathcal A|$ denoting the cardinality of the set $\mathcal A$,
$$
N_{+}(T) =|\mathcal S_+(T)|,  \quad
   N_{-}(T) =|\mathcal S_-(T)| .
$$
It follows from \eqref{PC} that almost all   zeros $\rho=\hf+i\g$ of the zeta-function are simple, that is, the number of them with ordinates in $(0, T]$ is $\sim N(T)$. Thus,  consecutive ordinates almost always alternate between the  two sets $\mathcal S_{+}(T)$ and $\mathcal S_{-}(T)$ and we have
\be\label{N pl,min asymp}
N_{+}(T)  \sim N_{-}(T)\sim \frac12 N(T)\qquad(T\to\infty).
\ee

Suppose now that  $\g$ is the ordinate of a simple zero
and that $\g^*$ is the next ordinate greater than $\g$.
Setting
$$
f(\alpha)\;:=\;\int_0^\alpha \biggl[1-\bigg(\frac{\sin\pi u}{\pi u}\bigg)^2\biggr] \d u
$$
with  $\a>0$,
we see from \eqref{PC} that
$$
 \sum_{\substack{0<\g,\g^* \leq T\\ \g^* -\g \leq \frac{2\pi \alpha}{\log T}}}  1
\;   \leq  \;\sum_{\substack{0<\g,\g' \leq T\\ 0< \gamma' -\gamma \leq \frac{2\pi \alpha}{\log T}} } 1
 \; \sim \; f(\alpha)\, N(T)\qquad(T\to\infty).
$$
Hence, the number of simple zeros $\rho=\hf+i\g$ with $0<\g\leq T$
and\, $\g^* -\g>2\pi\alpha/\log T$
\,   is greater than or equal to
$$ \big(1-f(\a)+o(1) \big) N(T)\qquad(T\to\infty).$$
By \eqref{N pl,min asymp} the number of these $\g$ that are  in $\mathcal S_+(T)$ (similarly, $\mathcal S_-(T)$) is therefore
$$
\geq \big(1-f(\a)+o(1) \big) N(T)-\tfrac12N(T) =\big(\tfrac12-f(\a)+o(1) \big)
N(T)\qquad(T\to\infty).
$$
Thus, if we define
$$
N_{+}(\a, T):=\sum_{\substack{\g\in \mathcal S_+(T) \\ \g^{*}-\g>2\pi \a/\log T}} 1,
$$
and $N_{-}(\a, T)$ similarly, then
\be\label{N_{pm}}
N_{\pm}(\a, T) \geq \big(\tfrac12-f(\a)+o(1) \big) N(T)\qquad(T\to\infty).
\ee

For $T$ large, let $B= B(T)>1$ be such that every gap  $\g^*-\g$
between consecutive ordinates of zeros with $\g\in (0, T]$ is
less than  $2\pi B/\log T$.
Then we have
\be\notag
\begin{split}
\int_{0}^{B} N_{+}(\a, T) \; \d \a
= & \int_{0}^{B}  \Bigg(\sum_{\substack{
\g\in \mathcal S_+(T) \\ \g^{*}-\g>2\pi \a/\log T}} 1  \Bigg)\d \a \\
=& \sum_{\substack{ \g\in \mathcal S_+(T)}}   \int_0^{((\g^{*}-\g) \log T)/2\pi } 1\;  \d \a   \\
=& \,\frac{\log T}{2\pi} \sum_{\substack{ \g\in \mathcal S_+(T)}}  (\g^{*}-\g)   \\
\le\, & \frac{\log T}{2\pi}\; \mu\big(I_{+}(0, T)\big).
\end{split}
\ee
Now $N_{+}(\a, T)$ is nonnegative,  so for  any   $A\in[0,B]$,
\be\notag
\mu\bigr(I_{+}(0,T)\bigl) \geq
 \frac{2\pi }{\log T}\; \int_{0}^{A} N_{+}(\a, T)  \d \a .
\ee
By  \eqref{N(T)} and \eqref{N_{pm}} we therefore see that
\be\notag
\mu\bigr(I_{+}(0,T)\bigl) \geq
T\; \int_{0}^{A} \big(\tfrac12-f(\a) \big)  \d \a +o(T)  \qquad(T\to\infty).
\ee
Using  Mathematica, we find  that the right-hand side  attains a maximum value slightly
  greater than $.32909\,T$ when $A\approx .952$. The same argument works
  {\it{mutatis mutandis}} for $\mu\bigr(I_{-}(0,T)\bigl)$, so
the proof of Theorem~\ref{thm: cond 1} is complete.

\section{Acknowledgement}
The authors wish to express their sincere gratitude to Dr. Christopher
Hughes for several invaluable discussions and for his  assistance with the computations in Tables 1 and 2. We   also  thank   Siegfred Baluyot  and   Fan Ge for their careful reading of the manuscript and comments. In particular, we are most grateful to them for a  suggestion that led to an improvement of the constant in Theorem~\ref{thm: cond 1}.

\vfill
\eject

\end{document}